\batchmode
\makeatletter
\def\input@path{{"/Users/russw/Documents/Research/mypapers/An algebraic groups perspective on Erdos-Ko-Rado/"}}
\makeatother
\documentclass[12pt,oneside,english,lowtilde]{amsart}
\usepackage[T1]{fontenc}
\usepackage[latin9]{inputenc}
\usepackage[a4paper]{geometry}
\usepackage{babel}
\usepackage{verbatim}
\usepackage{url}
\usepackage{enumitem}
\usepackage{amstext}
\usepackage{amsthm}
\usepackage{amssymb}
\usepackage[dvips,unicode=true,pdfusetitle,
 bookmarks=true,bookmarksnumbered=false,bookmarksopen=false,
 breaklinks=false,pdfborder={0 0 1},backref=false,colorlinks=false]
 {hyperref}
\usepackage{breakurl}

\makeatletter
\numberwithin{equation}{section}
\numberwithin{figure}{section}
\theoremstyle{plain}
\newtheorem{thm}{\protect\theoremname}[section]
\theoremstyle{plain}
\newtheorem{lem}[thm]{\protect\lemmaname}
\theoremstyle{remark}
\newtheorem{rem}[thm]{\protect\remarkname}
\theoremstyle{plain}
\newtheorem{prop}[thm]{\protect\propositionname}
\theoremstyle{plain}
\newtheorem{question}[thm]{\protect\questionname}
\theoremstyle{definition}
\newtheorem{example}[thm]{\protect\examplename}

\usepackage{lmodern}

\makeatother

\providecommand{\examplename}{Example}
\providecommand{\lemmaname}{Lemma}
\providecommand{\propositionname}{Proposition}
\providecommand{\questionname}{Question}
\providecommand{\remarkname}{Remark}
\providecommand{\theoremname}{Theorem}

\begin{document}
\global\long\def\Ind{\operatorname{Ind}}%

\global\long\def\Whis{\operatorname{Whis}}%

\global\long\def\Gr{\operatorname{Gr}}%

\global\long\def\init{\operatorname{init}}%

\global\long\def\st{\operatorname{star}}%

\global\long\def\del{\operatorname{del}}%

\global\long\def\link{\operatorname{link}}%

\global\long\def\image{\operatorname{image}}%

\global\long\def\cc{\mathbb{C}}%

\global\long\def\ff{\mathbb{F}}%

\global\long\def\pp{\mathbb{P}}%

\global\long\def\qq{\mathbb{Q}}%

\global\long\def\rr{\mathbb{R}}%

\global\long\def\zz{\mathbb{Z}}%

\global\long\def\normalin{\mathrel{\lhd}}%

\global\long\def\innormal{\mathrel{\rhd}}%

\global\long\def\semidirect{\mathbin{\rtimes}}%

\global\long\def\Stab{\operatorname{Stab}}%

\global\long\def\bdry{\partial}%

\global\long\def\susp{\operatorname{susp}}%

\global\long\def\lrprod{\mathop{\check{\prod}}}%

\global\long\def\lrtimes{\mathbin{\check{\times}}}%

\global\long\def\urtimes{\mathbin{\hat{\times}}}%

\global\long\def\urprod{\mathop{\hat{\prod}}}%

\global\long\def\subsetdot{\mathrel{\subset\!\!\!\!{\cdot}\,}}%

\global\long\def\dotsupset{\mathrel{\supset\!\!\!\!\!\cdot\,\,}}%

\global\long\def\precdot{\mathrel{\prec\!\!\!\cdot\,}}%

\global\long\def\dotsucc{\mathrel{\cdot\!\!\!\succ}}%

\global\long\def\des{\operatorname{des}}%

\global\long\def\rank{\operatorname{rank}}%

\global\long\def\height{\operatorname{height}}%

\global\long\def\modreln{\mathrel{M}}%

\global\long\def\link{\operatorname{link}}%

\global\long\def\freejoin{\mathbin{\circledast}}%

\global\long\def\stellarsd{\operatorname{stellar}}%

\global\long\def\conv{\operatorname{conv}}%

\global\long\def\disjointunion{\mathbin{\dot{\cup}}}%

\global\long\def\skel{\operatorname{skel}}%

\global\long\def\depth{\operatorname{depth}}%

\global\long\def\st{\operatorname{star}}%

\global\long\def\alexdual#1{#1^{\vee}}%

\global\long\def\reg{\operatorname{reg}}%

\global\long\def\shift{\operatorname{Shift}}%

\global\long\def\Homolred{\tilde{H}}%

\global\long\def\Dom{\operatorname{Dom}}%

\global\long\def\cosetposet{\overline{\mathfrak{C}}}%

\global\long\def\cosetlat{\mathfrak{C}}%

\global\long\def\ordcong{\mathcal{O}}%

\title{An algebraic groups perspective on Erd\H{o}s--Ko--Rado}
\author{Russ Woodroofe}
\address{Univerza na Primorskem, Glagoljaška 8, 6000 Koper, Slovenia}
\email{russ.woodroofe@famnit.upr.si}
\thanks{This work is supported in part by the Slovenian Research Agency (research
program P1-0285 and research projects J1-9108, N1-0160, J1-2451, J3-3003). }
\urladdr{\url{https://osebje.famnit.upr.si/~russ.woodroofe/}}
\begin{abstract}
We give a proof of the Erd\H{o}s--Ko--Rado Theorem using the Borel
Fixed Point Theorem from algebraic group theory. This perspective
gives a strong analogy between the Erd\H{o}s--Ko--Rado Theorem and
(generalizations of) the Gerstenhaber Theorem on spaces of nilpotent
matrices.
\end{abstract}

\maketitle

\section{\label{sec:Introduction}Introduction}

A family of sets is \emph{intersecting} if every pair of sets in the
family intersect nontrivially. The systematic study of intersecting
families of sets began with a 1961 paper of Erd\H{o}s, Ko, and Rado
\cite{Erdos/Ko/Rado:1961}, in which these authors characterized the
largest possible intersecting family of uniform sets.
\begin{thm}[Erd\H{o}s--Ko--Rado \cite{Erdos/Ko/Rado:1961}]
\label{thm:EKR} Suppose that $k\leq n/2$. If $\mathcal{A}$ is
an intersecting family of $k$-element subsets of $\left[n\right]$,
then $\left|\mathcal{A}\right|\leq{n-1 \choose k-1}$. If more strongly
$k<n/2$, then the equality $\left|\mathcal{A}\right|={n-1 \choose k-1}$
holds only if all sets in $\mathcal{A}$ share a common element.
\end{thm}

Theorem~\ref{thm:EKR}, while a non-trivial result, is especially
noted for admitting a large number of proofs. These tend to come in
one of several flavors: the original proof of \cite{Erdos/Ko/Rado:1961}
developed the idea of combinatorial shifting of families of sets,
the well-known proof of Katona \cite{Katona:1972} uses the $S_{n}$
symmetry of $\left[n\right]$ for double-counting, and there are several
proofs that are based on linear algebra \cite{Frankl/Furedi:2012,Furedi/Hwang/Weichsel:2006,Godsil/Meagher:2016}.

There are a number of generalizations of the Erd\H{o}s--Ko--Rado
theorem to different settings, often with different notions of intersecting.
Such theorems have been described as saying that the largest possible
construction is the obvious candidate.

Another family of results that say that the largest construction is
the obvious one comes from the world of nilpotent matrices. Early
results of this form were proved by Gerstenhaber.
\begin{thm}[Gerstenhaber \cite{Gerstenhaber:1958}; Serezhkin \cite{Serezhkin:1985}
removed a restriction on the field]
\label{thm:Gerstenhaber}Let $M_{n}$ be the vector space of $n\times n$
matrices over some field $\ff$. If $V$ is a vector subspace of $M_{n}$
consisting of nilpotent matrices, then $\dim V\leq{n \choose 2}$.
Moreover, equality holds only if $V$ is conjugate to the subalgebra
of $M_{n}$ consisting of strictly upper triangular matrices.
\end{thm}

It is interesting to remark that Gerstenhaber's work was published
at roughly the same time as Erd\H{o}s, Ko, and Rado published their
work. 

Gerstenhaber's work has been generalized to arbitrary Lie algebras
in work of Meshulam and Radwan \cite{Meshulam/Radwan:1998}, and of
Draisma, Kraft, and Kuttler \cite{Draisma/Kraft/Kuttler:2006}. The
Lie algebra terminology in the following theorem will not be used
in the remainder of the paper, and the unfamiliar reader may pass
over it lightly.
\begin{thm}
\label{thm:LieNilpotentSubspace}Let $\mathfrak{g}$ be a complex
semi-simple Lie algebra. If $V$ is a vector subspace of $\mathfrak{g}$
consisting of elements having nilpotent adjoint transformation, then
\begin{enumerate}
\item $\dim V\leq\frac{1}{2}\left(\dim\mathfrak{g}-\rank\mathfrak{g}\right)$
\cite{Meshulam/Radwan:1998}, and
\item equality holds only if $V$ is the nilradical of a Borel subalgebra
of $\mathfrak{g}$ \cite{Draisma/Kraft/Kuttler:2006}.
\end{enumerate}
\end{thm}

In another recent generalization, Sweet and MacDougall \cite{Sweet/MacDougall:2009}
found (using only elementary techniques) the maximal dimension of
a space of nilpotent matrices of nilpotence degree $2$.

In the current paper, we prove the following result, which generalizes
the inequality of Theorem~\ref{thm:EKR} and is directly analogous
to instances of that in Theorems~\ref{thm:Gerstenhaber} and \ref{thm:LieNilpotentSubspace}.
Let $\Lambda\cc^{n}$ be the exterior algebra over the vector space
$\cc^{n}$, and let $e_{1},e_{2},\dots,e_{n}$ be the standard basis
for $\cc^{n}$. For a subset $S\subseteq\Lambda\cc^{n}$, we write
$S\wedge S$ for $\left\{ x\wedge y:x,y\in S\right\} $.
\begin{thm}
\label{thm:ExtEKR}Let $V$ be a vector subspace of $\Lambda^{k}\cc^{n}$
such that $V\wedge V=0$. If $k\leq n/2$, then $\dim V\leq{n-1 \choose k-1}$. 
\end{thm}

Theorem~\ref{thm:EKR} obviously follows immediately from Theorem~\ref{thm:ExtEKR}
by associating with each $A\in\mathcal{A}$ the monomial $m_{A}$
in $\Lambda^{k}\cc^{n}$ that is supported by $A$, and letting $V=\left\langle m_{A}:A\in\mathcal{A}\right\rangle $.
Theorem~\ref{thm:ExtEKR} was first proved in the recent paper \cite{Scott/Wilmer:2021},
which shows it to follow from Theorem~\ref{thm:EKR}. 

Our proof of Theorem~\ref{thm:ExtEKR} here will instead use the
Borel Fixed-Point Theorem from the theory of algebraic groups, and
will be similar to the approach of \cite{Draisma/Kraft/Kuttler:2006}.
The resulting proof has a shifting-theoretic feel to it, and there
are relationships with combinatorial and algebraic shifting, as we
shall explain in Section \ref{sec:Shifting}.

It is natural to ask whether there is a proof of the structural part
of Theorem~\ref{thm:EKR} that is based on the Borel Fixed-Point
Theorem. Indeed, one could reasonably hope for such a proof of the
following stronger result:
\begin{thm}[Hilton--Milner \cite{Hilton/Milner:1967}]
\label{thm:HiltonMilner} For $2\leq k\leq n/2$, if $\mathcal{A}$
is an intersecting family of $k$-element subsets of $\left[n\right]$,
then $\left|\mathcal{A}\right|\leq{n-1 \choose k-1}-{n-k-1 \choose k-1}+1$
unless all sets in $\mathcal{A}$ share a common element.
\end{thm}

I don't know whether a Borel Fixed-Point Theorem proof of Theorem~\ref{thm:HiltonMilner}
is possible, but will discuss possible approaches and obstructions
in Section~\ref{sec:TowardsAlgHM}.

The paper is organized as follows. In Section~\ref{sec:Background},
which is rather long, we discuss all the background material needed
from algebraic geometry and combinatorics. In Section~\ref{sec:ProofMain},
which is quite short, we give the algebraic group theory proof of
Theorem~\ref{thm:ExtEKR}. In Section~\ref{sec:Shifting}, we discuss
the relationship of the algebraic groups perspective with the techniques
of combinatorial and algebraic shifting. We finish in Section~\ref{sec:TowardsAlgHM}
with a discussion of possible extensions of Theorem~\ref{thm:HiltonMilner}
to the exterior algebra situation.

\section*{Acknowledgements}

Thanks to Roya Beheshti, Jan Draisma and to Mathematics Stackoverflow
user Lazzaro Campeotti \cite{Stackexchange} for helping me with the
algebraic geometry background, particularly with understanding how
to define a subvariety of the Grassmannian from a projective variety.
Thanks to Claude Roché for pointing out the relevant work of de Rham.
Thanks to Alex Scott and Elizabeth Wilmer for discussing the relationship
with their work. 

\section{\label{sec:Background}Background}

\subsection{Exterior algebras and intersecting sets}

The \emph{exterior algebra} $\Lambda\cc^{n}$ is an anticommutative
analogue of the algebra of polynomials in $n$ variables. More specifically,
$\Lambda\cc^{n}$ is the $\mathbb{C}$-algebra generated by $e_{1},\dots,e_{n}$
with product $\wedge$, and subject to the square relation $x\wedge x=0$
for $x\in\cc^{n}$. The square relation yields the anticommutative
relation $x\wedge y=-y\wedge x$ for $x,y\in\cc^{n}$. The exterior
algebra is a graded algebra, and the $k$th homogenous component $\Lambda^{k}\cc^{n}$
consists of all elements of homogenous degree $k$, that is, all linear
combinations of wedge products of $k$ of the $e_{i}$ generators.

An \emph{(exterior) monomial} in $\Lambda\cc^{n}$ has the form $\alpha e_{i_{1}}\wedge e_{i_{2}}\wedge\dots\wedge e_{i_{k}}$
for some $\left\{ i_{1},\dots,i_{k}\right\} \subseteq\left[n\right]$
and $\alpha\in\cc$. In this situation, we say that the monomial is
\emph{supported} by $\left\{ i_{1},\dots,i_{k}\right\} $. Thus, monomials
in $\Lambda\cc^{n}$ are in correspondence up to scalar multiplication
with subsets of $\left[n\right]$. Since the product of two monomials
is $0$ if and only if the corresponding subsets intersect, the exterior
algebra is well-known to be a useful model for systems of intersecting
sets. See for example \cite[Chapter 6]{Babai/Frankl:1992}.

Extending from a set system to a system of elements from $\Lambda\cc^{n}$
has the advantage that we extend the group that naturally acts on
our system. Indeed, the group $GL_{n}=GL_{n}(\cc)$ acts on the vector
space $\left\langle e_{1},\dots,e_{n}\right\rangle \cong\Lambda^{1}\cc^{n}$,
and this action extends naturally to an action on each homogenous
component $\Lambda^{k}\cc^{n}$, hence to $\Lambda\cc^{n}$.

There is a close relation between annihilation (that is, elements
having product $0$) and factorization in the exterior algebra. A
useful form of this was observed by de Rham, and later rediscovered
by Dibag. A \emph{linear factor} of an element $v\in\Lambda\cc^{n}$
is an $a\in\Lambda^{1}\cc^{n}$ so that $v=a\wedge w$ for some $w\in\Lambda\cc^{n}$. 
\begin{lem}[\cite{deRham:1954,Dibag:1974}]
 \label{lem:Dibag} An element $v\in\Lambda^{k}\cc^{n}$ has $a\in\Lambda^{1}\cc$
as a linear factor if and only if $a\wedge v=0$.
\end{lem}

\subsection{Algebraic groups and shifted systems}

The group $GL_{n}$ is an example of an \emph{algebraic group}, since
its multiplication and addition operations can be expressed coordinate-wise
by polynomials. Subgroups of $GL_{n}$ that are given by the zeros
of polynomials (in some precise sense) are also algebraic groups.
All subgroups of $GL_{n}$ that we discuss here are algebraic.

We will use the following fundamental theorem from linear algebraic
groups and algebraic geometry, which may be found in numerous textbooks
\cite{Humphreys:1975,Milne:2017,Wallach:2017}. A\emph{ projective
algebraic variety} is a subset of projective space $\mathbb{P}^{n}\cong\mathbb{C}^{n+1}/\sim$
(where $\sim$ identifies points differing by a non-zero scalar multiple)
given by the zeros of a finite family of homogenous polynomials in
$n+1$ variables. Given a vector space $V$, we write $\pp(V)$ for
the projective space obtained by identifying non-zero scalar multiples.
Thus, for example $\pp^{n}=\pp(\cc^{n+1})$. 
\begin{thm}[Borel Fixed-Point Theorem]
\label{thm:BorelFPT} If $X\neq\emptyset$ is a projective algebraic
variety over an algebraically closed field, and $G$ is a connected,
solvable, linear algebraic group acting on $X$ by morphisms, then
there is a point in $X$ that is fixed by the action of $G$.
\end{thm}

Here, a \emph{morphism }between projective varieties is a map given
by homogenous polynomials of the same degree on the projective coordinates.
It is well-known that $GL_{n}$ is an algebraic group acting by morphisms
on projective space, and that this restricts to an action on any projective
variety that is closed under the action \cite[Lecture 10]{Harris:1995}.

In order to apply Theorem~\ref{thm:BorelFPT}, we need a connected
solvable subgroup of $GL_{n}$. Such a subgroup is provided by the
subgroup $B_{n}<GL_{n}$ of all (weakly) upper-triangular matrices.
Moreover, $B_{n}=T_{n}\rtimes U_{n}$, where $T_{n}$ consists of
all diagonal invertible matrices and $U_{n}$ of all upper triangular
matrices with $1$'s on the diagonal. On the other hand, the permutation
matrices also form a subgroup $W_{n}$ of $GL_{n}$, and $W_{n}$
is isomorphic to the symmetric group $S_{n}$. A relationship between
these subgroups is given by $GL_{n}=B_{n}W_{n}B_{n}$.
\begin{rem}
Although we will not need this fact, the maximal connected solvable
subgroups (the so-called \emph{Borel subgroups}) of $GL_{n}$ are
exactly the conjugates of $B_{n}$. We mention also that in the further
theory of linear algebraic and Lie groups, the subgroup $B_{n}$ is
called a \emph{Borel subgroup}, $T_{n}$ is a \emph{maximal torus},
and $W_{n}$ is a \emph{Weyl group}.
\end{rem}

There is a well-known relationship between fixed points of the action
of $B_{n}$ and $T_{n}$ on $\Lambda\cc^{n}$ and combinatorics of
set systems. A family $\mathcal{A}$ of subsets of $\left[n\right]$
is said to be \emph{shifted} if whenever $i>j$ and $S\in\mathcal{A}$
are such that $i\in S$ but $j\notin S$, then also $\left(S\setminus i\right)\cup j\in\mathcal{A}$. 
\begin{prop}[see e.g. \cite{Herzog/Hibi:2011,Miller/Sturmfels:2005}]
\label{prop:TnBnFixed}Let $V$ be a subspace of $\Lambda^{k}\cc^{n}$.
\begin{enumerate}
\item If $V$ is fixed by the action of $T_{n}$, then $V$ has a basis
consisting of monomials. 
\item If $V$ is fixed by the action of $B_{n}$, then $V$ has a basis
consisting of monomials whose supports form a shifted family of sets.
\end{enumerate}
\end{prop}

\begin{proof}
First, that $V$ is fixed by $T_{n}$ means that we may independently
scale $e_{1},\dots,e_{n}$ and remain in $V$. Now if $V$ has a basis
element $b$ that is the sum of at least two monomials, we may find
an $e_{i}$ that is in some of these monomials but not others. Multiplying
this $e_{i}$ by $-1$ allows us to replace $b$ by an element supported
by a smaller number of monomials. An easy inductive argument gives
that a $T_{n}$-fixed space has a basis of monomials. We remark that
such a basis is obviously unique.

For the second part, since $T_{n}\subseteq B_{n}$, we may assume
that we have a basis consisting of monomials. If $S$ is the support
of a monomial in $V$ with $i\in S$ and $j\notin S$ for $i>j$,
then the matrix $g$ sending $e_{i}$ to $e_{i}+e_{j}$ (and fixing
all other basis elements of $\cc^{n}$) is upper-triangular. Thus
$g\cdot m_{S}=m_{S}+m_{S\setminus i\cup j}$ is also in $V$, and
so $m_{S\setminus i\cup j}$ is in $V$, hence (by uniqueness) in
the monomial basis for $V$. 
\end{proof}

\subsection{\label{subsec:ProjVarieties}Exterior algebras and projective varieties}

The family of all $m$-dimensional subspaces of a vector space $X$
forms a projective variety, the \emph{Grassmannian} $\Gr_{m}(X)$.
The proof proceeds by identifying an $m$-dimensional subspace $Y$
spanned by $y_{1},\dots,y_{m}$ with $y_{1}\wedge\cdots\wedge y_{m}$
(up to a scalar multiple) in $\pp(\Lambda^{m}X)$. That the Grassmannian
is a projective variety follows from showing that the elements of
$\Lambda^{m}X$ that can be written as a product of elements from
$\Lambda^{1}X$ can be identified as the zeros of a system of polynomial
equations \cite{Harris:1995,Humphreys:1975}.

We will consider $m$-dimensional vector subspaces of $\Lambda^{k}\cc^{n}$.
It is perhaps amusing to note that the projective variety $\Gr_{m}(\Lambda^{k}\cc^{n})$
of such subspaces sits in $\pp\left(\Lambda^{m}\left(\Lambda^{k}\cc^{n}\right)\right)$.

It is straightforward to see that the condition in $\Lambda^{k}\cc^{n}$
that $v\wedge w=0$ is given by polynomial equations. This can be
extended to show that the condition that $V\wedge V=0$ yields a subvariety
of $\Gr_{m}(\Lambda^{k}\cc^{n})$, as follows. In \cite[Example 6.19]{Harris:1995},
it is shown that if $X$ is a projective subvariety of $\pp^{n}$,
then the set $\left\{ V\in\Gr_{m}:V\subseteq X\right\} $ is a subvariety.
The proof goes by constructing homogenous polynomial functions $f_{1},\dots,f_{\ell}$
on $\Gr_{m}$ so that if $V\in\Gr_{m}$, then $f_{1}(V),\dots,f_{\ell}(V)$
span $V$. Since the composition of polynomials is a polynomial, it
follows that $\left\{ V\in\Gr_{m}:V\subseteq X\right\} $ is identified
as the zeros of the so-composed polynomials. The same argument on
pairs of elements in the spanning set shows that $\left\{ V\in\Gr_{m}(\Lambda^{k}\cc^{n}):V\wedge V=0\right\} $
is the zero set of a system of polynomials.

\subsection{\label{subsec:EKRforShifted}Erd\H{o}s--Ko--Rado for shifted set
systems}

For $k=n/2$, Theorem~\ref{thm:EKR} is trivial, since a set and
its complement may not both be in $\mathcal{A}$, and as ${2k-1 \choose k-1}=\frac{1}{2}{2k \choose k}$. 

For $k<n/2$, if we make the additional assumption that the family
$\mathcal{A}$ in Theorem~\ref{thm:EKR} is shifted, then the proof
is an easy induction. Decompose $\mathcal{A}$ as the disjoint union
of the family $\st_{\mathcal{A}}n$ consisting of sets in $\mathcal{A}$
with $n$ as an element, and its complement $\del_{\mathcal{A}}n=\mathcal{A}\setminus\st_{\mathcal{A}}n$.
Let $\link_{\mathcal{A}}n=\left\{ A\setminus n:A\in\st_{\mathcal{A}}n\right\} $.
Then $\link_{\mathcal{A}}n$ and $\del_{\mathcal{A}}n$ are clearly
also shifted, and $\del_{\mathcal{A}}n$ is clearly intersecting.

Now if $C,D\in\link_{\mathcal{A}}n$ have $C\cap D=\emptyset$, then
(since $k\leq n/2$) there is some $i\neq n$ in $[n]\setminus(C\cup D)$.
But then $C\cup i$ and $D\cup n$ are nonintersecting sets in $\mathcal{A}$
by shiftedness, a contradiction. It follows that $\link_{\mathcal{A}}n$
is intersecting.

Now by induction, we have $\left|\mathcal{A}\right|=\left|\link_{\mathcal{A}}n\right|+\left|\del_{\mathcal{A}}n\right|\leq{n-2 \choose k-2}+{n-2 \choose k-1}={n-1 \choose k-1}$.

\section{\label{sec:ProofMain}Proof of the main theorem}

Having set up a large amount of algebraic machinery, the proof of
Theorem~\ref{thm:ExtEKR} now follows quickly. Indeed, if the variety
of $V\subseteq\cc^{n}$ of dimension $m$ with $V\wedge V=0$ is nonempty,
then there is a fixed point for the action of $B_{n}$ by Theorem~\ref{thm:BorelFPT},
hence a shifted family of $m$ intersecting $k$-sets by Proposition~\ref{prop:TnBnFixed}.
That $m\leq{n-1 \choose k-1}$ now follows by the Erd\H{o}s--Ko--Rado
Theorem for shifted set systems (as in Section~\ref{subsec:EKRforShifted}).

\section{\label{sec:Shifting}Shifting and limits of algebraic group actions}

\subsection{\label{subsec:LimitingMatrixGp}Generalizing combinatorial shifting
via limits of matrix group actions}

Combinatorial shifting may be realized via limits of actions of matrix
subgroups, as we describe below in Lemma~\ref{lem:GeomShiftExtsubspace},
Proposition~\ref{prop:ShiftedFromLimiting}, and the surrounding
discussion. A similar relationship in a somewhat different setting
was previously discussed by Knutson \cite[Section 3]{Knutson:2014UNP},
as we will review. A completely different take on the relationship
between combinatorial shifting and algebra is given by Murai and Hibi
\cite[Section 2]{Murai/Hibi:2009}.

We consider the parametrized family of linear transformations $M_{ij}(t)$
given by the matrix that is $1$ on the diagonal, $t$ at the $j,i$
entry, and $0$ elsewhere. Indeed, $M_{ij}(t)$ is an injective homomorphism
$\cc^{+}\to GL_{n}$. 
\begin{rem}
Similar homomorphisms are referred to as \emph{one-parameter subgroups}
in the Lie algebra literature. However, we caution that the algebraic
geometry and algebraic groups literature tends to reserve this term
for homomorphisms from the multiplicative group of $\cc$ (rather
than from the additive group). As a result, we avoid the term.
\end{rem}

The action of $M_{ij}(t)$ on an element $v$ in a projective variety
has a limiting value $\lim_{t\to\infty}M_{ij}(t)\cdot v$. As $M_{ij}(s)\cdot\lim_{t\to\infty}M_{ij}(t)\cdot v=\lim_{t\to\infty}M(s+t)\cdot v=\lim_{t\to\infty}M(t)\cdot v$,
the limit point is preserved under the action by $M_{ij}(t)$.

We consider the limiting behavior of $M_{ij}(t)$ first on $\pp(\cc^{n})$,
and then extend to related varieties. The interesting behavior for
the action on $\pp(\cc^{n})$ occurs in the action on $e_{i}$, which
is sent to 
\[
e_{i}+te_{j}\sim\frac{1}{t}e_{i}+e_{j}\to0+e_{j}.
\]
Similar rescaling arguments show that the limiting action fixes the
hyperplane consisting of all vectors with zero $e_{i}$ component,
and sends all other points to $e_{j}$. 

We now extend to the action on $\Gr_{m}(\cc^{n})$. The action of
$M_{ij}(t)$ on each vector is as in the preceding paragraph. But
we notice that if a subspace $V$ contains (for example) both $e_{i}$
and $e_{j}$, then a Gaussian elimination argument gives that $M_{ij}(t)\cdot V=V$.
If $V$ contains $e_{i}$ and not $e_{j}$, such an elimination cannot
be carried out, and $\lim_{t\to\infty}M_{ij}(t)\cdot V$ replaces
$e_{i}$ with $e_{j}$ in a basis for $V$. More generally:
\begin{lem}[{Knutson \cite[Lemma 3.4]{Knutson:2014UNP}}]
\label{lem:GeomShiftVS} If $V$ is an $m$-dimensional subspace
of $\cc^{n}$ (i.e., $V\in\Gr_{m}(\cc^{n}))$, then 
\[
\lim_{t\to\infty}M_{ij}(t)\cdot V=\begin{cases}
V & \text{if }V\subseteq\left\langle e_{1},\dots,\hat{e}_{i},\dots,e_{n}\right\rangle \text{ or }e_{j}\in V,\\
\left(V\cap\left\langle e_{1},\dots,\hat{e}_{i},\dots,e_{n}\right\rangle \right)\oplus\left\langle e_{j}\right\rangle  & \text{otherwise.}
\end{cases}
\]
\end{lem}

\begin{proof}
It is obvious by the preceding discussion that if $V\subseteq\left\langle e_{1},\dots,\hat{e}_{i},\dots,e_{n}\right\rangle $
then $V$ is fixed in the limit of the action, and that otherwise
$e_{j}$ and $V\cap\left\langle e_{1},\dots,\hat{e}_{i},\dots,e_{n}\right\rangle $
are contained in $\lim_{t\to\infty}M_{ij}(t)\cdot V$. If $V\cap\left\langle e_{1},\dots,\hat{e}_{i},\dots,e_{n}\right\rangle +e_{j}$
is $m$-dimensional, then this characterizes $\lim_{t\to\infty}M_{ij}(t)\cdot V$.

Otherwise, we have $e_{j}\in V$. In this case, we can reduce $e_{i}+te_{j}$
to $e_{i}$ in each $M_{ij}(t)\cdot V$, so that $M_{ij}(t)\cdot V=V$
for each value of $t$. The result follows.
\end{proof}
The situation of Lemma~\ref{lem:GeomShiftVS} is not quite what we
are interested in. Rather, we are interested in the limit action induced
on $\pp(\Lambda^{k}\cc^{n})$, and on $\Gr_{m}(\Lambda^{k}\cc^{n})$.
The action on $\pp(\Lambda^{k}\cc^{n})$ should be clear. For ease
of notation, we consider the action of $M_{21}(t)$. Consider an element
of the form $x=e_{2}\wedge v+e_{1}\wedge e_{2}\wedge w+u$, where
$v$ and $w$ are in the subalgebra $\Lambda\left\langle e_{3},\dots,e_{n}\right\rangle $,
and $u$ is in $\Lambda\left\langle e_{1},e_{3}\dots,e_{n}\right\rangle $.
The transformation $M_{21}(t)$ sends $x$ to $(e_{2}+te_{1})\wedge v+e_{1}\wedge e_{2}\wedge w+u$.
In the limit and after renormalizing, this converges to $e_{2}\wedge v$
if that term is nonzero, and to $e_{1}\wedge e_{2}\wedge w+u$ otherwise. 

The limiting action on $\Gr_{m}(\Lambda^{k}\cc^{n})$ is induced from
that on $\pp(\Lambda^{k}\cc^{n})$ in a similar manner to that of
Lemma~\ref{lem:GeomShiftVS}.
\begin{lem}
\label{lem:GeomShiftExtsubspace}Let $V$ be an $m$-dimensional subspace
of $\Lambda^{k}\cc^{n}$ (i.e., $V\in\Gr_{m}(\Lambda^{k}\cc^{n})$),
and let $\varphi:V\to\Lambda^{k}\cc^{n}$ be the (singular) linear
map sending monomials of the form $e_{i}\wedge v$ to $e_{j}\wedge v$,
all others to $0$. Then 
\[
\lim_{t\to\infty}M_{ij}(t)\cdot V=\varphi(V)+\varphi^{-1}(V\cap\varphi(V)).
\]
 Notice that $V\cap\Lambda^{k}\left\langle e_{1},\dots,\hat{e}_{i},\dots,e_{n}\right\rangle \subseteq\varphi^{-1}(0)$.
\end{lem}

\begin{proof}
Without loss of generality, assume $i=2$ and $j=1$. As in Lemma~\ref{lem:GeomShiftVS},
it is clear from the discussion of the action of $M_{21}(t)$ on $\pp(\Lambda^{k}\cc^{n})$
that $\varphi(V)$ is contained in $\lim_{t\to\infty}M_{21}(t)\cdot V$.
An element is in $V\cap\varphi(V)$ when it is of the form $e_{1}\wedge v$,
and is $\varphi(e_{2}\wedge v+y)$ for some $y\in\Lambda^{k}\left\langle e_{1},e_{3},\dots,e_{n}\right\rangle +e_{1}\wedge e_{2}\wedge\Lambda^{k}\left\langle e_{3},\dots,e_{n}\right\rangle $.
In this situation, $M_{21}(t)\cdot(e_{1}\wedge v+y)=e_{1}\wedge v+te_{2}\wedge v+y$,
and we can use the $e_{1}\wedge v$ element of $V$ to ``row-reduce''
to $e_{2}\wedge v+y$. Thus, the right-hand side is contained in the
left-hand side.

We now notice that, since $\varphi^{2}$ is the zero map, the intersection
between the two terms in the right-hand sum is $\varphi^{-1}(0)$.
It now follows from elementary linear algebra  that the dimension
of the sum on the right-hand side is $m$, completing the proof.
\end{proof}
Recall that the \emph{combinatorial shift} $S_{ij}$ of a set system
$\mathcal{A}$ replaces each set $A\in\mathcal{A}$ containing $i$
with $\left(A\setminus i\right)\cup j$ if the latter set is not already
present, and leaves $A$ alone otherwise. The original proof of Theorem~\ref{thm:EKR}
was by combinatorial shifting, and the technique has seen much use
since; see \cite{Frankl:1987} for a survey. It follows immediately
from Lemma~\ref{lem:GeomShiftExtsubspace} that if $V$ has a basis
of monomials supported by the set system $\mathcal{A}$, then $\lim_{t\to\infty}M_{ij}(t)\cdot V$
is supported by $S_{ij}(\mathcal{A})$. Thus, combinatorial shifting
of a set system is realized by a limiting action of an algebraic group.

Conversely, we have the following. 
\begin{prop}
\label{prop:ShiftedFromLimiting}Let $V$ be a subspace of $\Lambda^{k}\cc^{n}$.
\begin{enumerate}
\item If $V=\lim_{t\to\infty}M_{ij}(t)\cdot V$ for some given $i,j$, then
$V$ is fixed by the action of $M_{ij}(t)$.
\item If $V=\lim_{t\to\infty}M_{ij}(t)$ for all $i>j$, then $V$ is fixed
by the action of $B_{n}$.
\end{enumerate}
\end{prop}

\begin{proof}
The fixed point behavior of (1) holds for any action of $M_{ij}(t)$
on a projective variety.

In the situation of (2), it follows from (1) that $V$ is fixed under
all upper triangular matrices with $1$'s on the diagonal. It remains
to show that $V$ is fixed by diagonal matrices. As projective monomials
are fixed by diagonal matrices, this is equivalent by Proposition~\ref{prop:TnBnFixed}
to showing that $V$ has a basis of monomials. But if $V$ has a basis
element $b$ that is supported by at least two monomials, then we
may find $i>j$ so that some monomials contain $e_{i}$ but not $e_{j}$
and vice-versa. Then $\varphi(b)\neq0$ is in $0$, has a smaller
support, and can be used to reduced $b$. A straightforward induction
gives that $V$ is generated by monomials, as desired.
\end{proof}
We see a variant on the algebraic groups-based proof of Theorem~\ref{thm:ExtEKR},
as follows. By Proposition~\ref{prop:ShiftedFromLimiting}~(2),
repeatedly applying limiting actions of $M_{ij}(t)$ for $i>j$ yields
a fixed point of the action of $B_{n}$. Now Proposition~\ref{prop:TnBnFixed}
and Section~\ref{subsec:EKRforShifted} give the desired result.

\subsection{\label{subsec:RelateAlgShift}Diagonal matrix actions, with a relationship
to algebraic shifting}

Another technique that has been used for proving Erd\H{o}s--Ko--Rado
type theorems \cite{Fakhari:2017,Woodroofe:2011a} is that of algebraic
shifting. Algebraic shifting uses generic initial ideal techniques
(related to Gröbner bases) to produce a shifted set system from a
set system, and indeed, a shifted simplicial complex from a simplicial
complex. An overview may be found in \cite{Kalai:2002} or in \cite{Herzog/Hibi:2011}.
The connection between algebra and shiftedness again comes from Borel-fixed
ideals, although the Borel-fixed property does not directly arise
from a group action in the typical presentation of this material.

Algebraic shifting has excellent theoretical properties, but it is
not so easy to make computations with it. In comparison, Theorem~\ref{thm:BorelFPT}
allows relatively direct examination of orbits, so long as they can
be grouped together into varieties. 

It is well-known to experts in the field that it is also possible
to describe algebraic shifting via limiting actions of $GL_{n}$.
We briefly survey this approach, as it doesn't seem to be as broadly
known as it deserves. Consider the diagonal matrix $N(t)$ with entries
$t^{-2^{1}},t^{-2^{2}},\dots,t^{-2^{n}}$. Thus, the action of $N(t)$
on $\Lambda\cc^{n}$ weights each of the $2^{n}$ monomials of $\Lambda\cc^{n}$
by a distinct power of $t$, where the powers of $t$ arise from the
standard bijection between subsets of $[n]$ and binaries sequences
of length $n$. It is clear that lexicographically earlier subsets
have a higher weighting.

An entirely similar argument to those in the previous section (via
projective rescaling) yields that for $v\in\pp(\Lambda\cc^{n})$,
we have $\lim_{t\to\infty}N(t)\cdot v$ to be the monomial in $v$
whose support is lexicographically earliest. We call this monomial
the \emph{initial monomial} of $v$. 
\begin{rem}
Similar ideas are studied in the commutative algebra literature under
the name of \emph{initial ideals}. We refer the reader to e.g. \cite{Herzog/Hibi:2011,Miller/Sturmfels:2005}
for an overview. 
\end{rem}

Applying similar arguments to a vector space, we obtain:
\begin{lem}
\label{lem:limitingtorus}If $V$ is an $m$-dimensional subspace
of $\Lambda^{k}\cc^{n}$ (i.e., $V\in\Gr_{m}(\Lambda^{k}\cc^{n}))$,
then $\lim_{t\to\infty}N(t)\cdot V$ is the subspace $\init(V)$ generated
by the initial monomials of a basis for $V$.
\end{lem}

\begin{proof}
It follows from the above discussion that $\init(V)\subseteq\lim_{t\to\infty}N(t)\cdot V$.
Now straightforward linear algebra gives that $\init(V)$ is spanned
by the initial monomials of a basis for $V$, giving that $\init(V)$
is $m$-dimensional. The result follows.
\end{proof}
As the framework of algebraic shifting is based upon taking an initial
ideal with respect to a generic basis, Lemma~\ref{lem:limitingtorus}
and similar results can be used to give a description of algebraic
shifting from the algebraic groups perspective. 

\section{\label{sec:TowardsAlgHM}Towards an exterior analogue of the Hilton--Milner
Theorem}

Having given an algebraic groups-based proof of Theorem~\ref{thm:EKR},
it would be interesting to give a similar proof of Theorem~\ref{thm:HiltonMilner}.
Indeed, it is natural to ask the following question:
\begin{question}
\label{que:ExteriorHM}Let $V$ be a subspace of $\Lambda^{k}\cc^{n}$
satisfying $V\wedge V=0$. If the dimension of $V$ is ${n-1 \choose k-1}$
(or possibly larger than ${n-1 \choose k-1}-{n-k-1 \choose k-1}+1$)),
then must there be an $a\in\Lambda^{1}\cc^{n}\cong\cc^{n}$ so that
$a\wedge V=0$?
\end{question}

Scott and Wilmer also ask the ${n-1 \choose k-1}$ case of Question~\ref{que:ExteriorHM}
in \cite[Section 2.2]{Scott/Wilmer:2021}.

It is worthwhile to remark that, by Lemma~\ref{lem:Dibag}, it is
equivalent to ask whether there is a common linear factor\emph{ }of
$V$. That is, is there (under the conditions of the question) a fixed
$a\in\cc^{n}$ so that every $v\in V$ may be written as $a\wedge w$
for some $w\in\Lambda^{k-1}\cc^{n}$? 

A natural approach to this question is to try to imitate the argument
of \cite{Draisma/Kraft/Kuttler:2006}, possibly leavened with the
shifting-based proofs of Theorem~\ref{thm:HiltonMilner} \cite{Frankl:1987,Frankl/Furedi:1986}.
A key step of the approach in \cite{Draisma/Kraft/Kuttler:2006} is
to choose the basis with respect to which our matrices are upper-triangular.
Their argument proceeds by showing that every maximum dimensional
vector space of nilpotent matrices (or more generally Lie algebra
elements) contains a matrix which is upper triangular with respect
to a unique choice of basis. This is done by showing that the set
of matrices that are upper triangular with respect to multiple bases
form an algebraically closed set, and applying Theorem~\ref{thm:BorelFPT}
to get a contradiction. 

The analogue would be to show that under some additional condition,
the space $V$ of Theorem~\ref{thm:EKR} has an element with a unique
linear factor. Indeed, it appears likely that an analogue of the argument
of \cite{Draisma/Kraft/Kuttler:2006} can be used to show that the
subspace of $\Lambda^{k}\cc^{n}$ consisting of the elements having
more than one linear factor is closed. 

One obstacle to following this path is that there are spaces $V$
with $V\wedge V=0$ but which have many elements with no linear factor.
Indeed, one can find such a $V$ that is spanned by elements each
of which has no linear factor!
\begin{example}
\label{exa:CrossExteriorExample}Let $k$ be odd, and let $n=2k$.
Let $\mathcal{A}$ be the set of all $k$-subsets of $[n]$ containing
$1$. Then $\mathcal{A}$ is obviously an intersecting family. It
is easy to see that the family of complements of the sets in $\mathcal{A}$
also forms an intersecting family. Let $V$ be spanned by the monomials
$m_{A}+m_{A^{c}}$ over $A\in\mathcal{A}$. (Thus, for $k=3$, one
such monomial is $e_{1}\wedge e_{2}\wedge e_{3}+e_{4}\wedge e_{5}\wedge e_{6}$.)
Now since $k$ is odd, the exterior product square of any such element
is $0$, while the product of $m_{A}+m_{A^{c}}$ and $m_{B}+m_{B^{c}}$
is $0$ by an intersection argument. Now the multiplication map $\wedge(m_{A}+m_{A^{c}})$
sends the generators $e_{i}$ to a linearly independent subset of
$\Lambda^{k+1}\cc^{n}$;  applying Lemma~\ref{lem:Dibag} shows that
no element in the spanning set has any linear factor.
\end{example}

Of course, Example~\ref{exa:CrossExteriorExample} has $n=2k$ and
so does not satisfy the dimension bound suggested by Theorem~\ref{thm:HiltonMilner},
but it illustrates one difficulty in answering Question~\ref{que:ExteriorHM}. 

Difficulties also arise in attempting to generalize the shifting-based
approach of \cite{Frankl:1987,Frankl/Furedi:1986}. An intersecting
set system with no common intersection may be transformed by shifting
operations into a shifted system with the same properties. In the
situation of Question~\ref{que:ExteriorHM}, can the techniques of
Section~\ref{subsec:LimitingMatrixGp} be used to do the same?

The techniques of this paper are applicable to other intersection
problems in extremal set theory, so long as the condition corresponds
to a subvariety in $\Gr_{m}(\Lambda^{k}\cc^{n})$. For example, Seyed
Amin Seyed Fakhari has suggested {[}private communication{]} that
replacing the pairs of exterior elements in Section~\ref{subsec:ProjVarieties}
with $s$-tuples of exterior elements may yield an algebraic groups
approach to the Erd\H{o}s Matching Conjecture. 

\bibliographystyle{hamsplain}
\bibliography{2_Users_russw_Documents_Research_Master,3_Users_russw_Documents_Research_mypapers_An_al____perspective_on_Erdos-Ko-Rado_Stackexchange}

\end{document}